\documentclass[11pt, reqno]{amsart}
\usepackage{a4wide}
\usepackage{hyperref}
\usepackage{amsmath}
\usepackage{paralist}
\usepackage{amssymb}
\usepackage{amsthm}
\usepackage{amscd}
\usepackage{graphicx,mathrsfs}
\usepackage{fontenc}
\usepackage{enumitem}
\renewcommand{\subjclassname}{2020 Mathematics Subject Classification}
\begin{document}
\title[Memory Interactions Model]
{Asymptotic Flocking Behavior in Particle and Kinetic Systems with biological Memory Term}
\author[ Yishuo Wang, Yawei Wei]
{Yishuo Wang, Yawei Wei}

\address{Yawei Wei \newline
School of Mathematical Sciences and LPMC\\ Nankai University\\ Tianjin 300071, China}
\email{weiyawei@nankai.edu.cn}

\address{Yishuo Wang \newline
School of Mathematical and Sciences \\ Nankai University\\ Tianjin 300071, China}
\email{1120240033@mail.nankai.edu.cn}

\thanks{Acknowledgements: This work is supported by the NSFC under the grands 12271269, and  the Fundamental Research Funds for the Central Universities.}



\begin{abstract}
We investigate a Cucker-Smale type flocking model with a biological memory term and conclude that memory makes flocking harder. The model introduces $h(t)$ to describe memory, with dynamics $\frac{dh}{dt}=x(t)+v(t)-\lambda h(t)$ reflecting path, velocity dependence and exponential decay. In microscopic level, the flocking phenomenon becomes conditional, requires communication kernel $\psi(t)$ has polynomial lower bound and memory deviation $H(t)$ decays exponentially. In mesoscopic level, flocking rates exhibit exponential or algebraic decay and the range of communication strength $\beta$ become smaller.
\end{abstract}

\maketitle
\numberwithin{equation}{section}
\newtheorem{theorem}{theorem}[section]
\newtheorem{remark}{remark}[section]
\newtheorem{lemma}{Lemma}[section]
 \newtheorem{corollary}{Corollary}[section]
\newtheorem{assumption}{Assumption}[section]
\newtheorem{definition}{Definition}[section]
\newtheorem{proposition}{proposition}[section]
\newcommand{\R}{\mathbb{R}}
\newcommand{\B}{\mathbf{B}}
\newcommand{\e}{{\mathcal {E}}}
\newcommand{\M}{{\mathcal{H}}}
\newcommand{\Y}{{\mathbf{P}}}

\renewcommand{\subjclassname}
\allowdisplaybreaks

\section{Introduction}

In this paper, we study the dynamic Cucker-Smale model with a biological memory:
\begin{equation}\label{1.1}
\begin{cases}
\frac{dx^{i}}{dt}=v^{i},\\
\frac{dv^{i}}{dt}=\frac{1}{N}\displaystyle\sum_{j=1}^{N}r(x^{i},x^{j})(v^{j}-v^{i})-h^{i},\\
\frac{dh^{i}}{dt}=x^{i}(t)+v^{i}(t)-\lambda h^{i}(t),\;\;\lambda>0.
\end{cases}
\end{equation}

Here for each $i=1,2,\cdots,N$, $x^{i}(t)=(x_{1}^{i}(t),\cdots,x_{d}^{i}(t))\in\R^{d}$ denotes the position of agent $i$, $v^{i}(t)=(v_{1}^{i}(t),\cdots,v_{d}^{i}(t))\in\R^{d}$ denotes the velocity of agent $i$.  The communication kernel $r(x,y)$ is symmetric, bi-particle, non-increasing, Lipschitz continues and there exists a constant $A$ such that
\begin{equation}\label{1.2}
r(x,y)=r(y,x)\leq A,\;\;for\;x,y\in\R^{d}.
\end{equation}

In \eqref{1.1}, we introduce $h^{i}(t)=(h_{1}^{i}(t),\cdots,h_{d}^{i}(t))\in\R^{d}$ for each agent $i$ to represent the memory of agent depends on the previous path and velovity. The memory term $h^{i}(t)$ has the following convolution form
\begin{equation}\label{1.3}
h^{i}(t)=h^{i}(0)e^{-\lambda t}+\int_{0}^{t}(x^{i}(s)+v^{i}(s))e^{\lambda(s-t)}ds,\;\;i=1,2,\cdots,N.
\end{equation}

In both natural and artificial systems, many groups (such as bird flocks \cite{Tadmor2017From,Toner1998Flocks}, fish schools \cite{Degond2008Large,Parrish2002Self}, bacteria \cite{Topaz2004Swarming}, drone formations \cite{He2017Feedback}, etc.) can achieve global coordinated motion through local individual communications, a phenomenon known as flocking behavior.  The research on flocking behavior was first proposed by Viscek \cite{Degond2008Continuum,Vicsek2006Novel}, who studied the communication between individuals and adjacent individuals within a fixed range. While subsequent extensions like the Cucker-Smale (C-S) model \cite{CS2007,Cucker2007On} introduced nonlocal, distance-dependent communications:
\begin{equation}\label{1.4}
\dot{x}^{i}=v^{i},\;\;\;\dot{v}^{i}=\frac{1}{N}\displaystyle\sum_{j=1}^{N}r(x^{i},x^{j})(v^{j}-v^{i})
\end{equation}

In recent works \cite{Carrillo2010Asymptotic,Carrillo2017A,Motsch2011A,Choi2017Emergent} further research on flocking behavior with different communication kernels. It is assumed that individuals make decisions only based on the current state of their neighbors. However, in real biological systems, the behavior of individual is often influenced by its path experiences and historical velocities. The historical states is described by time delay terms as \cite{Pignotti2017Flocking}.

For example, fish will remember their foraging path, safe zone, predator location and swimming speed \cite{Brown2003Social,Biro2016Bringing}; migratory animals will follow their previous migration paths and previous speed to migrate \cite{Odling2003influence}, and so on. Caplan and Kato found that the brain reduces the autocorrelation of input signals through retinal dog transform and other methods, so that natural istic stimuli (such as paths and images) can be efficiently encoded by convolutional memory as \cite{Caplan2019}. Therefore, based on the work above all, we studies the C-S model with memory term \eqref{1.1} to considering the influenced of historical experiences. The memory term $h^{i}(t)$ has convolutional form and acts as an internal driving force beyond the communications between biological individuals, directly affecting the individual's velocity dynamics equation.

This paper is organized as follows. The discussion results of the flocking phenomenon in the microscopic level is presented in Section 2. In section 3, we derive the mesoscopic memory model and prove its existence. Finally in section 4, we discuss the time-asymptotic flocking behavior of system in mesoscopic level.

\section{Main results}
 Let $(x^{c}(t),v^{c}(t),h^{c}(t))$ be the center of system \eqref{1.1},
\begin{equation}\label{2.1}
x^{c}(t)=\frac{1}{N}\displaystyle\sum_{i=1}^{N}x^{i}(t),\;\;v^{c}(t)=\frac{1}{N}\displaystyle\sum_{i=1}^{N}v^{i}(t),\;\;h^{c}(t)=\frac{1}{N}\displaystyle\sum_{i=1}^{N}h^{i}(t).
\end{equation}

We construct the potential energy function, kinetic energy function and memory deviation function
\begin{equation*}
X(t)=\displaystyle\sum_{i=1}^{N}\vert x^{i}(t)-x^{c}(t)\vert^{2},\;V(t)=\displaystyle\sum_{i=1}^{N}\vert v^{i}(t)-v^{c}(t)\vert^{2},\;H(t)=\displaystyle\sum_{i=1}^{N}\vert h^{i}(t)-h^{c}(t)\vert^{2}.
\end{equation*}

Let $m_{j}(t),j=1,2$ represent the total momentum and energy of the system \eqref{1.1},
\begin{displaymath}
m_{1}(t):=\displaystyle\sum_{i=1}^{N}v^{i}(t),\;\;m_{2}(t):=\frac{1}{2}\displaystyle\sum_{i=1}^{N}\vert v^{i}(t)\vert^{2}.
\end{displaymath}

\begin{proposition}\label{proposition 2.1}Let $(x^{i}(t),v^{i}(t),h^{i}(t))$ be the solution of the system \eqref{1.1}. We have the following properties
\begin{equation}\label{2.2}
\frac{d}{dt}m_{1}(t)= -Nh^{c}(t),
\end{equation}
\begin{equation}\label{2.3}
\frac{d}{dt}m_{2}(t)\leq\sqrt{V(t)}\sqrt{H(t)}-Nv^{c}\cdot h^{c}.
\end{equation}
\end{proposition}

\begin{proposition}\label{proposition 2.2}[Central system] Let $(x^{i}(t),v^{i}(t),h^{i}(t))$ be the solution of the system \eqref{1.1}. When $\lambda>1$, there exist constants $C>0$, $\delta>0$ such that
\begin{equation}\label{2.4}
\vert x^{c}(t)\vert+\vert v^{c}(t)\vert+\vert h^{c}(t)\vert\leq Ce^{-\delta t}.
\end{equation}
\end{proposition}

For the convenience of discussion, we use the standard C-S communication kernel \cite{Carrillo2010Particle} here,
\begin{equation*}
r(x,y)=\frac{A}{(1+\vert x-y\vert^{2})^{\beta}},\;\;\;A>0,\;\;\beta\geq0,\;\;for\;x,y\in\R^{d},
\end{equation*}
and the flocking phenomenon depends on the minimum form of communication kernel
\begin{equation*}
\psi(t):=\displaystyle\min_{1\leq i,j\leq N}r(x^{i}(t),x^{j}(t)).
\end{equation*}

\begin{theorem}\label{theorem 2.3}[Conditional flocking] Let $(x^{i}(t),v^{i}(t),h^{i}(t))$ be the solution of the system \eqref{1.1}. We make the following assumptions:
\begin{item}
\item[(1)]The minimum form of communication kernel $\psi(t)$ satisfies a polynomial lower bound, there exist constants $C>0$ and $0\leq\beta'<1$ such that
\begin{equation}\label{2.5}
\psi(t)\geq\frac{C}{(1+t)^{\beta'}}\;\;\text{satisfy}\;\;\int_{0}^{\infty}\psi(s)ds=\infty
\end{equation}
\item[(2)](Memory deviation integrability condition) There exists $\gamma>0$, such that $H(t)$ satisfies
\begin{equation}\label{2.6}
\int_{0}^{\infty}e^{\gamma t^{1-\beta'}}\sqrt{H(t)}dt<\infty
\end{equation}
\end{item}

Then we have
\begin{equation}\label{2.7}
\lim_{t\to\infty}\vert v^{i}(t)-v^{c}(t)\vert=0,\;\;i=1,\cdots,N,
\end{equation}
and
\begin{equation}\label{2.8}
\sup_{t\geq0}\vert x^{i}(t)-x^{c}(t)\vert<\infty,\;\;i=1,\cdots,N.
\end{equation}
\end{theorem}

\begin{remark}\label{remark 2.4}
In \cite{Brown2003Social}, it is shown that fish share memory through social learning to maintain group structure and synchronous behavior. This suggests that memory differences between individuals should diminish over time to achieve flocking. In our model, this corresponds to the requirement that the memory deviation $H(t)$ decays exponentially as in \eqref{2.6}. Therefore, the flocking phenomenon not only depends on the communication kernel but also on the decay rate of memory deviation.
\end{remark}

After discussing the flocking phenomenon of microstates, we derived the Vlasov type equation using the BBGKY hierarchy \cite{Benedetto1999A,Russo1996Kinetic} and mean field limit \cite{Golse2016On,Poyato2016Euler} to describe the mesoscopic state of memory systems.

Since the memory variable does not increase the effective dimension of the phase space, we assume that the number of agents is large enough to enable us to analyze their distribution functions:
\begin{equation*}
f^{N}=f^{N}(x^{1},v^{1},\cdots,x^{N},v^{N},t),\;\;(x^{i},v^{i})\in~\R^{d}\times~\R^{d}.
\end{equation*}

Then we use the BBGKY hierarchy, based of the Liouville equation \cite{Kennard1938},
\begin{equation}\label{2.9}
\small
\begin{cases}
\partial _{t}f^{N}+\displaystyle\sum_{i=1}^{N}v^{i}\cdot\nabla_{x^{i}}f^{N}+\frac{1}{N}\displaystyle\sum_{i=1}^{N}\nabla_{v^{i}}\cdot\bigl[\displaystyle\sum_{j=1}^{N}(r(x^{i},x^{j})(v^{j}-v^{i})-h^{i}(t))f^{N}\bigl]=0,\\
\frac{dh^{i}}{dt}=x^{i}(t)+v^{i}(t)-\lambda h^{i}(t),\;\;\lambda>0.
\end{cases}
\end{equation}

Consider the marginal distribution $f^{N-1}(x^{1},v^{1},t)$ and $f^{N-2}(x^{1},v^{1},x^{2},v^{2},t)$:
\begin{align*}
f^{N-1}(x^{1},v^{1},t): &=\int_{\R^{2d(N-1)}}f^{N}dx^{2}dv^{2}\cdots dx^{N}dv^{N},\\
f^{N-2}(x^{1},v^{1},x^{2},v^{2},t) :&=\int_{\R^{2d(N-2)}}f^{N}dx^{3}dv^{3}\cdots dx^{N}dv^{N}.
\end{align*}

Integrate $dx^{2}dv^{2}\cdots dx^{N}dv^{N}$ on both sides of \eqref{2.9}. Then we obtain
\begin{equation}\label{2.10}
\partial _{t}f^{N-1}+v^{1}\cdot\nabla_{x^{1}}f^{N-1}+F(x^{1},v^{1},t)-G(x^{1},v^{1},t)=0,
\end{equation}
where
\begin{align*}
F(x^{1},v^{1},t):&=\frac{1}{N}\int_{\R^{2d(N-1)}}\displaystyle\sum_{i=1}^{N}\nabla_{v^{i}}\cdot[\displaystyle\sum_{j=1}^{N}(r(x^{i},x^{j})(v^{j}-v^{i})f^{N}]dx^{2}dv^{2}\cdots dx^{N}dv^{N},\\
G(x^{1},v^{1},t):&=\frac{1}{N}\int_{\R^{2d(N-1)}}(\displaystyle\sum_{i=1}^{N}h^{i}(t)\cdot\displaystyle\sum_{i=1}^{N}\nabla_{v^{i}}f^{N})dx^{2}dv^{2}\cdots dx^{N}dv^{N}.
\end{align*}

We take the mean-field limit $N\to\infty$ to obtain the equation of the mesoscopic state:
\begin{equation}\label{2.11}
\begin{cases}
\partial_{t}f+v^{1}\cdot\nabla_{x^{1}}f+\nabla_{v^{1}}\cdot F(f)-z(t)\cdot\nabla_{v^{1}}f=0,\\
F(f)=\int_{\R^{2d}}r(x^{1},x^{2})(v^{2}-v^{1})gdx^{2}dv^{2},\\
z'(t)=\overline{x}(t)+\overline{v}(t)-\lambda z(t),\;\;\lambda>0,
\end{cases}
\end{equation}
where $\overline{x}(t)=\int_{\R^{2d}}x^{1}f(x^{1},v^{1},t)dx^{1}dv^{1}$, $\overline{v}(t)=\int_{\R^{2d}}v^{1}f(x^{1},v^{1},t)dx^{1}dv^{1}$,\\ $f(x^{1},v^{1},t)=\displaystyle\lim_{N\to\infty}f^{N-1}(x^{1},v^{1},t)$,
$ g(x^{1},v^{1},x^{2},v^{2},t)=\displaystyle\lim_{N\to\infty}f^{N-2}(x^{1},v^{1},x^{2},v^{2},t)$.

Finally, based on the assumption of molecular chaos \cite{Brown2008}, we get
\begin{equation*}
g(x^{1},v^{1},x^{2},v^{2},t)=f(x^{1},v^{1},t)f(x^{2},v^{2},t).
\end{equation*}

For simplicity of notation, we henceforth denote $(x^{1},v^{1})=(x,v)$ and $(x^{2},v^{2})=(y,w)$, then we get the Vlasov-type mean-field model with memory term:
\begin{equation}\label{2.12}
\begin{cases}
\partial_{t}f+v\cdot\nabla_{x}f+\nabla_{v}\cdot [(F(f)-z(t))f]=0,\\
F(f)=\int_{\R^{2d}}r(x,y)(w-v)f(y,w,t)dydw,\\
z'(t)=\overline{x}(t)+\overline{v}(t)-\lambda z(t),\;\;\lambda>0.
\end{cases}
\end{equation}

Here, $f(x,v,t)$ represents the mesoscopic density of the system, and
\begin{equation*}
\overline {x}(t)=\displaystyle\int_{\R^{2d}}xf(x,v,t)dxdv,\;\;\overline {v}(t)=\displaystyle\int_{\R^{2d}}vf(x,v,t)dxdv
\end{equation*}
represent the central mass of the system and the total momentum of system. $z(t)$ is the average memory strength. For the memory communication system of mesoscopic states, we provide a definition of the system solution and prove the existence of the solution.

\begin{definition}For given $T>0$, we say that $(f,z)$ is a weak solution of Vlasov-type equation \eqref{2.12} if the following conditions are satisfied:
\begin{item}
\item[(1)]the initial density $f_{0}\in L_{+}^{1}(\R^{2d})\cap L^{\infty}(\R^{2d})$,
\item[(2)]for all $\eta\in C^{\infty}_{c}(\R^{d}\times \R^{d}\times[0,T])$ with $\eta(x,v,T)=0$,
\begin{equation}\label{2.13}
\small
\int_{R^{2d}}f_{0}\eta(x,v,0)dxdv+\int_{0}^{T}\int_{\R^{2d}}f(\partial_{t}\eta+v\cdot\nabla_{x}\eta+F(f)\cdot\nabla_{v}\eta-z(t)\cdot\nabla_{v}\eta)dxdvdt=0.
\end{equation}
\item[(3)]for every $t\in[0,T]$,
\begin{equation}\label{2.14}
z'(t)=\overline{x}(t)+\overline{v}(t)-\lambda z(t),\;\;z(0)=z_{0},
\end{equation}
with the mesoscopic states defined by
\begin{equation*}
\overline{x}(t)=\int_{\R^{2d}}xf(x,v,t)dxdv,\;\;\overline{v}(t)=\int_{\R^{2d}}vf(x,v,t)dxdv.
\end{equation*}
\end{item}
\end{definition}

\begin{theorem}\label{theorem 2.6}Let $T>0$. Suppose that $f_{0}$ satisfies
\begin{equation*}
f_{0}\in L_{+}^{1}(\R^{2d})\cap L^{\infty}(\R^{2d}),\;\;\vert v\vert^{2}f_{0}\in L^{1}(\R^{2d})\;\;and\;\;z_{0}\in\R^{d}
\end{equation*}
Then there exists a weak solution $(f,z)$ of the Vlasov-type equation \eqref{2.12} in the sense of definition.
\end{theorem}

Similar to the discussion of microscopic level, we discuss the momentum and energy of mesoscopic level.

Let
\begin{equation*}
M_{0}=\int_{\R^{2d}}fdxdv
\end{equation*}
be the total mass. We find that by discussing  the memory decay rate $\lambda$ of mesoscopic systems , we can obtain the minimum value form of the communication kernel.

\begin{proposition}\label{proposition 2.7}Let $(f,z)$ be the global solution in Theorem \ref{theorem 2.6}. Then there exist a constant $\overline k_{0}$, such that for any $(x(t),y(t))\in\R^{d}\times\R^{d}$, $\psi(t)=\displaystyle\min_{(x(t),y(t))} r(x-y)$ and $\lambda>1$ satisfies
\begin{equation}\label{2.15}
\psi(t)\geq
\frac{A}{\overline k_{0}}(1+t^{2}+t^{6})^{-\beta},\;\;\overline k_{0}>0.
\end{equation}
\end{proposition}

\begin{remark}\label{remark 2.8}
In the standard C-S model \cite{Tadmor2017From}, when mesoscopic asymptotic flocking occurs, the minimum value form of the communication kernel is
\begin{equation*}
\psi(t)\geq\frac{A}{\overline\kappa}(1+t^{2}+t^{4})^{-\beta},\;\;\overline\kappa>0.
\end{equation*}
With the introduction of memory terms, we can still obtain a form similar to the standard C-S model.
\end{remark}

Finally, we define the energy fluctuation function of the system
\begin{equation}\label{2.16}
u_{c}(t):=\frac{1}{M_{0}}\int_{\R^{2d}}vf(x,v,t)dxdv,\;\;\Lambda[f(t)]:=\int_{\R^{2d}}\vert v-u_{c}\vert^{2}f(x,v,t)dxdv.
\end{equation}

\begin{theorem}\label{theorem 2.9}[Unconditional flocking] Let $(f,z)$ be the global solution of the Vlasov-type flocking equation \eqref{2.12}. When $\lambda>1$, the decay of its energy fluctuation around the mean bulk velocity $u_{c}$, is given as following
\begin{equation}\label{2.17}
\Lambda[f(t)]\leq
\begin{cases}
\Lambda[f_{0}]e^{-\gamma_{1}t^{1-6\beta}},\;\;\;\;\beta\in[0,\frac{1}{6}),\\
\Lambda[f_{0}](1+t)^{-\gamma_{2}},\;\;\;\;\beta=\frac{1}{6},
\end{cases}
\end{equation}
where $\gamma_{1}$ and $\gamma_{2}$ are positive constants.
\end{theorem}

\begin{remark}\label{remark 2.10}
In the standard C-S model \cite{Tadmor2017From}, the critical value of $\beta$ is $\frac{1}{2}$. With the addition of memory terms, we find that the range of communication strength $\beta$ become smaller. The flocking phenomenon is difficult to occurrence. Memory makes flocking harder.
\end{remark}

\begin{remark}\label{remark 2.11}
We find that the occurrence of flocking phenomenon is not only depend on $\beta$, but also depend on memory decay rate $\lambda$. Through our analysis, we find that the memory decay should strong enough to keep the flocking, without any addition lower-bound assumptions about the communication kernel. And the condition $\lambda>1$ is sharp. When $\lambda\leq1$, the center system loses stability, the flocking phenomenon will not occurrence in mesoscopic level.
\end{remark}

The main contributions of this paper are summarized as follows. First, we construct a multi particle dynamics model with biological memory mechanism, which extends the classical Cucker-Smale model. The ODE system \eqref{1.1} involves position, velocity and memory states. This model introduces biological memory, which directly acts on the individual's acceleration, so as to more comprehensively simulate the biological decision-making process. The kinetic equation $\frac{dh}{dt}=x(t)+v(t)-\lambda h(t)$ of the memory $h(t)$ has a clear biological explanation as \cite{Xie2020Convolution}. The individual current position $x(t)$ and current velocity $v(t)$ are used as the continuous input memory. The term $-\lambda h(t)$ describes the natural forgetting process of memory, and its decay rate is controlled by parameter $\lambda$. This structure can continuously reflect the impact of individual path history on current behavior. Second, in microscopic level, we find that the flocking behavior is difficult to occurrence. It requires the communication kernel has a polynomial lower bound form. And the memory deviation needs to satisfy the integrability condition as \cite{Brown2003Social}. Third, in mesoscopic level, memory makes flocking harder. We find that compared with classical C-S model, the range of communication strength $\beta$ become smaller. With the addition of memory term, the occurrence of flocking phenomenon requires memory decay strong enough without the extra conditions of $\psi(t)$ and $H(t)$.

\section{The memory system of microscopic states}

We simplify symbols of the system, let
\begin{equation*}
\small x(t)=(x^{1}(t),\cdots,x^{N}(t))^{T},\;v(t)=(v^{1}(t),\cdots,v^{N}(t))^{T},\;h(t)=(h^{1}(t),\cdots,h^{N}(t))^{T},
\end{equation*}
and the initial values are
\begin{equation*}
\small x(0)=(x^{1}(0),\cdots,x^{N}(0))^{T},\;v(0)=(v^{1}(0),\cdots,v^{N}(0))^{T},\;h(0)=(h^{1}(0),\cdots,h^{N}(0))^{T}.
\end{equation*}
\begin{lemma}\label{Lemma 3.1}Let $(x(0), v(0), h(0))\in\R^{N\times d}\times\R^{N\times d}\times\R^{N\times d}$ be the initial value of $(x,v,h)$, there exists a unique solution $(x^{i}(t),v^{i}(t),h^{i}(t))$, $i=1,2,\cdots,N$ of system \eqref{1.1}, defined for all $t\in(0,+\infty)$.
\end{lemma}

\begin{proof}
The existence of solutions for the above ordinary differential systems, similar to the existence of solutions for standard C-S systems, can be obtained through fixed point theorems.

Write the system in autonomous ODE form
\begin{equation*}
\frac{dY}{dt}=F(Y,t)=
\begin{bmatrix}
v\\
R-h\\
x+v-\lambda h
\end{bmatrix},
\end{equation*}
where \small$Y=[x(t),v(t),h(t)]$ and \small$R=(\frac{1}{N}\displaystyle\sum_{j=1}^{N}r(x^{1},x^{j})(v^{j}-v^{1}),\cdots,\frac{1}{N}\displaystyle\sum_{j=1}^{N}r(x^{N},x^{j})(v^{j}-v^{N}))^{T}$.

Similar to the standard C-S model, since we have added a linear term, we do not alter the Lipschitz property of $F(Y,t)$. Combined with the extension theorem, we can obtain the existence and uniqueness of the solution for ordinary differential system (\ref{1.1}).
\end{proof}

\subsection{Proof of Proposition \ref{proposition 2.1}}

The proof of the proposition is relatively simple, mainly using Cauchy-Schwarz inequalities.

By using \eqref{1.2}, we have
\begin{equation*}
\displaystyle\sum_{1\leq i,j\leq N}r(x^{i},x^{j})(v^{j}-v^{i})=0.
\end{equation*}

Then we have

\begin{equation*}
\begin{split}
\frac{d}{dt}m_{1}(t)  =\frac{1}{N}\displaystyle\sum_{1\leq i,j\leq N}r(x^{i},x^{j})(v^{j}-v^{i})-\displaystyle\sum_{i=1}^{N}h^{i} = -\displaystyle\sum_{i=1}^{N}h^{i} = -Nh^{c}(t).
\end{split}
\end{equation*}

Moreover, we have
\begin{equation*}
\displaystyle\sum_{1\leq i,j\leq N}r(x^{i},x^{j})v^{i}\cdot (v^{j}-v^{i})=-\displaystyle\sum_{1\leq i,j\leq N}r(x^{i},x^{j})v^{j}\cdot (v^{j}-v^{i}).
\end{equation*}

Then the kinetic energy dissipation is as following
\begin{equation*}
\begin{split}
\frac{d}{dt}m_{2}(t) & =\displaystyle\sum_{i=1}^{N}v^{i}\cdot \frac{d}{dt}v^{i}=-\frac{1}{2N}\displaystyle\sum_{1\leq i,j\leq N}r(x^{i},x^{j})(v^{i}-v^{j})^{2}-\displaystyle\sum_{i=1}^{N}v^{i}\cdot h^{i}\\
&\leq-\displaystyle\sum_{i=1}^{N}v^{i}\cdot h^{i}=-\displaystyle\sum_{i=1}^{N}(v^{i}-v^{c})(h^{i}-h^{c})-Nv^{c}\cdot h^{c}\\
&\leq\sqrt{V(t)}\sqrt{H(t)}-Nv^{c}\cdot h^{c}.
\end{split}
\end{equation*}

\subsection{Proof of Proposition \ref{proposition 2.2}}

From the definition of the center system \eqref{2.1} and system \eqref{1.1}, we obtain
\begin{equation}\label{3.1}
\begin{cases}
\frac{dx^{c}}{dt}=v^{c},\\
\frac{dv^{c}}{dt}=-h^{c},\\
\frac{dh^{c}}{dt}=x^{c}+v^{c}-\lambda h^{c}.
\end{cases}
\end{equation}

Therefore, we can obtain the characteristic equation about $x^{c}(t)$
\begin{equation}\label{3.2}
\mu^{3}+\lambda\mu^{2}+\mu+1=0.
\end{equation}

By using Routh-Hurwitz criterion, when $\lambda>1$ the solution of the characteristic equation is
\begin{equation*}
\mu_{1}<0,\;\;\mu_{2}=\alpha+i\beta,\;\;\mu_{3}=\alpha-i\beta,
\end{equation*}
where $\alpha<0$ and $\beta>0$.

Then we can obtain the solution of $x^{c}(t)$
\begin{equation}\label{3.3}
x^{c}(t)=C_{1}e^{\mu_{1}t}+e^{\alpha t}(C_{2}\cos(\beta t)+C_{3}\sin(\beta t)),
\end{equation}
where $\mu_{1}<0$, $\alpha<0$, $\beta>0$, $C_{1}$, $C_{2}$, $C_{3}$ are constants.

Similarly, it can be concluded that $v^{c}(t)$ and $h^{c}(t)$ have the same form.

Hence, there exist constants $C>0$, $\delta>0$ such that
\begin{equation}\label{3.5}
\vert x^{c}(t)\vert+\vert v^{c}(t)\vert+\vert h^{c}(t)\vert\leq Ce^{-\delta t},\;\;t>0.
\end{equation}\hfill$\square$

\subsection{Flocking process}
\begin{lemma}\label{lemma 3.2}Let $(x^{i}(t),v^{i}(t),h^{i}(t))$ be the solution of the system \eqref{1.1}. Then $V(t)$ satisfies the differential inequality
\begin{equation}\label{3.6}
\frac{d}{dt}V(t)\leq -2\psi(t)V(t)+2\sqrt{V(t)}\sqrt{H(t)}.
\end{equation}
\end{lemma}

\begin{proof}
By using \eqref{1.2}, we find that
\begin{equation*}
\begin{split}
\frac{d}{dt}V(t) &= -\frac{1}{N}\displaystyle\sum_{1\leq i,j\leq N}r(x^{i},x^{j})(v^{i}-v^{j})^{2}-2\displaystyle\sum_{i=1}^{N}(v^{i}-v^{c})\cdot h^{i}\\
&\leq-2\psi(t)\displaystyle\sum_{i=1}^{N}(v^{i}-v^{c})^{2}-2\displaystyle\sum_{i=1}^{N}(v^{i}-v^{c})\cdot(h^{i}-h^{c})\\
&\leq -2\psi(t)V(t)+2\sqrt{V(t)}\sqrt{H(t)}.
\end{split}
\end{equation*}

Then there exists
\begin{equation*}
\frac{d}{dt}V(t)\leq -2\psi(t)V(t)+2\sqrt{V(t)}\sqrt{H(t)}.
\end{equation*}
\end{proof}

\begin{lemma}\label{lemma 3.3}Let $(x^{i}(t),v^{i}(t),h^{i}(t))$ be the solution of the system \eqref{1.1}. Then $H(t)$ satisfies the differential inequality
\begin{equation}\label{3.7}
\frac{d}{dt}H(t)\leq-2\lambda H(t)+2\sqrt{H(t)}\sqrt{X(t)}+2\sqrt{H(t)}\sqrt{V(t)}.
\end{equation}
\end{lemma}

\begin{proof}
Since we have
\begin{equation*}
\frac{dh^{i}}{dt}=x^{i}+v^{i}-\lambda h^{i},\;\;\frac{dh^{c}}{dt}=x^{c}+v^{c}-\lambda h^{c}.
\end{equation*}

Then we can obtain
\begin{equation*}
\begin{split}
\frac{d}{dt}H(t) =& 2\displaystyle\sum_{i=1}^{N}(h^{i}-h^{c})\cdot\frac{d}{dt}(h^{i}-h^{c})\\
=& 2\displaystyle\sum_{i=1}^{N}(h^{i}-h^{c})\cdot(x^{i}-x^{c})\\
&+2\displaystyle\sum_{i=1}^{N}(h^{i}-h^{c})\cdot(v^{i}-v^{c})-2\lambda\displaystyle\sum_{i=1}^{N}(h^{i}-h^{c})^{2}\\
\leq& -2\lambda H(t)+2\sqrt{H(t)}\sqrt{X(t)}+2\sqrt{H(t)}\sqrt{V(t)}.
\end{split}
\end{equation*}
\end{proof}

\begin{lemma}\label{lemma 3.4}Let $(x^{i}(t),v^{i}(t),h^{i}(t))$ be the solution of the system \eqref{1.1}. Then $X(t)$ satisfies the differential inequality
\begin{equation}\label{3.8}
X(t)\leq 2X_{0}+\frac{1}{2}(\int_{0}^{t}\sqrt{V(s)}ds)^{2},\;\;\;\;t>0.
\end{equation}
\end{lemma}

\begin{proof}
By using Cauchy-Schwartz's inequality, we have
\begin{equation*}
\frac{d}{dt}X(t) = 2\displaystyle\sum_{i=1}^{N}(x^{i}-x^{c})\cdot (v^{i}-v^{c})\leq2\sqrt{V(t)}\sqrt{X(t)}.
\end{equation*}

Then the solution of this differential inequality yields
\begin{equation*}
X(t) \leq (\sqrt{X_{0}}+\frac{1}{2}\int_{0}^{t}\sqrt{V(s)}ds)^{2}\leq 2X_{0}+\frac{1}{2}(\int_{0}^{t}\sqrt{V(s)}ds)^{2}.
\end{equation*}
\end{proof}

\begin{remark}\label{remark 3.5}
Through the estimation of Lemma \ref{lemma 3.2}, Lemma \ref{lemma 3.3} and Lemma \ref{lemma 3.4}, we find that $V(t)$, $H(t)$ and $X(t)$ cannot form a closed system. So the occurrence of the flocking phenomenon has become more difficult and requires additional restrictions.
\end{remark}

\noindent\textbf{Proof of Theorem \ref{theorem 2.3}.}

Since we have
\begin{equation*}
\frac{d}{dt}V(t)\leq -2\psi(t)V(t)+2\sqrt{V(t)}\sqrt{H(t)}.
\end{equation*}

Let $Y(t)=\sqrt{V(t)}$, simplify the above equation, we can obtain
\begin{equation}\label{3.9}
\frac{d}{dt}Y(t)\leq -\psi(t)Y(t)+\sqrt{H(t)}.
\end{equation}

Integrate both sides of \eqref{3.9}, we have
\begin{equation}\label{3.10}
Y(t)\leq Y(0)e^{-\int_{0}^{t}\psi(s)ds}+\int_{0}^{t}e^{-\int_{s}^{t}\psi(\tau)d\tau}\sqrt{H(s)}ds.
\end{equation}

Since we have $\psi(t)\geq C(1+t)^{\beta'}$, then we can obtain
\begin{equation*}
Y(t)\leq e^{-c(1+t)^{1-\beta'}}(Y(0)+\int_{0}^{t}e^{c(1+s)^{1-\beta'}}\sqrt{H(s)}ds).
\end{equation*}

By using the memory deviation integrability assumption, we have
\begin{equation}\label{3.10}
Y(t)\leq Ce^{-c(1+t)^{1-\beta'}}.
\end{equation}

Then we can obtain
\begin{equation*}
\lim_{t\to\infty}Y(t)=\lim_{t\to\infty}V(t)=0.
\end{equation*}

Then we have
\begin{equation}\label{3.11}
\lim_{t\to\infty}\vert v^{i}(t)-v^{c}(t)\vert=0,\;\;i=1,\cdots,N.
\end{equation}

And for $\vert x^{i}(t)-x^{c}(t)\vert$, we have
\begin{equation*}
\begin{split}
\vert x^{i}(t)-x^{c}(t)\vert &\leq \vert x^{i}(0)-x^{c}(0)\vert+\int_{0}^{t}\vert v^{i}(s)-v^{c}(s)\vert ds\\
&\leq\vert x^{i}(0)-x^{c}(0)\vert+\int_{0}^{t}Y(s)ds
\end{split}
\end{equation*}

By using \eqref{3.10}, we have
\begin{equation*}
\int_{0}^{\infty}Y(s)ds<\infty.
\end{equation*}

Then there exists
\begin{equation}\label{3.12}
\sup_{t\geq0}\vert x^{i}(t)-x^{c}(t)\vert<\infty,\;\;i=1,\cdots,N.
\end{equation}

\hfill$\square$

\section{The memory system of mesoscopic states}

In the proof of existence, we mainly use the linear iteration methods \cite{Choi2024From}. We first linearize the equation \eqref{2.12}:
\begin{equation}\label{4.1}
\begin{cases}
\partial_{t}f^{n}+v\cdot\nabla_{x}f^{n}+\nabla_{v}\cdot[F(f^{n-1})f^{n}]-z^{n-1}(t)\cdot\nabla_{v}f^{n}=0,\\
(z^{n-1})'=\overline{x}^{n-1}(t)+\overline{v}^{n-1}(t)-\lambda z^{n-1}(t),\;\;\lambda>0.
\end{cases}
\end{equation}
with the initial data and first iteration step:
\begin{equation*}
z^{n}(t)\big\vert_{t=0}=z^{0},\;\;f^{n}(x,v,t)\big\vert_{t=0}=f^{0}(x,v)\;\;\mbox{for all }n\geq 1,
\end{equation*}
\begin{equation*}
f^{0}(x,v,t)=f^{0}(x,v)\;\;\mbox{for }(x,v,t)\in \R^{d}\times\R^{d}\times(0,T).
\end{equation*}

Here
\begin{equation*}
F(f^{n-1})=\int_{\R^{2d}}r(x,y)(w-v)f^{n-1}(y,w,t)dydw,
\end{equation*}
and
\begin{equation*}
\overline{x}^{n-1}(t)=\int_{\R^{2d}}xf^{n-1}(x,v,t)dxdv,\;\;\overline{v}^{n-1}(t)=\int_{\R^{2d}}vf^{n-1}(x,v,t)dxdv.
\end{equation*}

Rewrite the equation and write it in gradient form
\begin{equation}\label{3.2}
\partial_{t}f^{n}+A\cdot\nabla f^{n}=-f^{n}\nabla_{v}\cdot F(f^{n-1}),
\end{equation}
where
\begin{equation*}
\small
A=
\begin{pmatrix}
v\\
F(f^{n-1})-z^{n-1}(t)
\end{pmatrix}.
\end{equation*}

Then by the standard existence theory for transport equations, we know that the equation has a solution $f^{n}\in L^{\infty}(0,T;L^{1}(\R^{d}\times\R^{d}))$.

Next, in order to prove the convergence of the iterative equation, we need to make some prior estimates:\\
$\bullet$($\parallel f^{n}\parallel_{L^{1}}$ estimate):

 We find that integration of equation \eqref{4.1} against $(f^{n})^{p-1}$, we have
\begin{equation*}
\frac{1}{p}\frac{d}{dt}\int_{\R^{2d}}(f^{n})^{p}dxdv=(\frac{1}{p}-1)\int_{\R^{2d}}\nabla_{v}\cdot F(f^{n-1})(f^{n})^{p}dxdv.
\end{equation*}

When $p=1$, we have $\frac{d}{dt}\int_{\R^{2d}}(f^{n})^{p}dxdv=0$. Then $M_{0}(t)=\int_{\R^{2d}}(f^{n})^{p}dxdv=M_{0}$ is a constant, which means that the total mass is conserved in time.\\
$\bullet$($\vert z^{n-1}(t)\vert$ estimate):

Since $(z^{n-1})'=\overline{x}^{n-1}(t)+\overline{v}^{n-1}(t)-\lambda z^{n-1}(t)$, we have
\begin{equation}\label{4.3}
z^{n-1}(t)=z(0)e^{-\lambda t}+\int_{0}^{t}e^{-\lambda(t-s)}(\overline{x}^{n-1}(s)+\overline{v}^{n-1}(s))ds.
\end{equation}

Along the characteristic line of the equation \eqref{4.1}, we can obtain
\begin{equation*}
x'(t)=v(t),\;\;v'(t)=F(f^{n-1})(x,v,t)-z^{n-1}(t).
\end{equation*}

For the estimation of positional moments, we have
\begin{equation}\label{4.4}
\begin{split}
\vert\overline{x}^{n-1}(t)\vert &\leq\int_{\R^{2d}}(\vert x(0)\vert+\int_{0}^{t}\vert v(\tau)\vert d\tau)f^{n-1}dxdv\\
&\leq \vert x(0)\vert M_{0}+\int_{0}^{t}\int_{\R^{2d}}\vert v(\tau)\vert f^{n-1}dxdvd\tau.
\end{split}
\end{equation}

By using the characteristic line equation, we can obtain
\begin{equation}\label{4.5}
\overline{v}^{n-1}(t)\leq\overline{v}^{n-1}(0)+\int_{0}^{t}(C\overline{v}^{n-1}(s)+M_{0}\vert z^{n-1}(s)\vert)ds.
\end{equation}

Using the three inequalities \eqref{4.3}, \eqref{4.4} and \eqref{4.5} obtained above together to obtain an inequality system
\begin{equation}\label{4.6}
\begin{cases}
\vert\overline{x}^{n-1}(t)\vert\leq\vert x(0)\vert M_{0}+\int_{0}^{t}\bar{v}^{n-1}(\tau)d\tau,\\
\overline{v}^{n-1}(t)\leq\overline{v}(0)+\int_{0}^{t}(C\overline{v}^{n-1}(s)+M_{0}\vert z^{n-1}(s)\vert)ds,\\
\vert z^{n-1}(t)\vert\leq\vert z(0)\vert+\int_{0}^{t}\vert\overline{x}^{n-1}(s)\vert ds+\int_{0}^{t}\vert\overline{v}^{n-1}(s)\vert ds.
\end{cases}
\end{equation}

By solving the system of differential inequalities \eqref{4.6}, we can obtain
\begin{equation*}
\overline{v}^{n-1}(t)\leq C.
\end{equation*}

Then we have
\begin{equation}\label{4.7}
\vert z^{n-1}(t)\vert\leq M'
\end{equation}
where $M'$ is a constant consists with $T$, $M_{0}$, $C$, $\lambda$, $x(0)$, $v(0)$ and $z(0)$.
\\
$\bullet$($\vert\vert f^{n}\vert\vert_{L_{\infty}}$ estimate):

For the $\vert\vert f^{n}\vert\vert_{L_{\infty}}$estimate, we introduce the backward characteristic line method.

Let
\begin{equation*}
Z^{n}(r):=(X^{n}(r),V^{n}(r)):=(X^{n}(r;t,x,v),V^{n}(r;t,x,v)),
\end{equation*}
which solves
\begin{equation}\label{4.8}
X'(r)=V^{n}(r),\;\;V'(r)=F(f^{n-1})(Z^{n}(r),r)-z^{n-1}(r),
\end{equation}
with the terminal data $Z^{n}(s)=(x,v)$.

Along the characteristics, we have
\begin{equation*}
\begin{split}
\frac{d}{dr}f^{n}(Z^{n}(r),r) &=-f^{n}(Z^{n}(r),r)\nabla_{v}\cdot F(f^{n-1}(Z^{n}(r),r))\\
&=d\int_{\R^{2d}}r(x-y)f^{n-1}(y,w,r)f^{n}(Z^{n}(r),r)dydw.
\end{split}
\end{equation*}

Thus,
\begin{equation}\label{4.9}
f^{n}(x,v,t)=f^{n}(Z^{n}(0))\exp(-d\int_{0}^{t}\int_{\R^{2d}}r(x-y)f^{n-1}(y,w,r)dydwdr)\leq Cf^{0}.
\end{equation}
\\
$\bullet$(Kinetic energy estimate):

Integration of equation \eqref{4.1} against $v^{2}$, then we have
\begin{equation*}
\frac{d}{dt}\int_{\R^{2d}}v^{2}f^{n}dxdv = 2\int_{\R^{2d}}v\cdot F(f^{n-1})f^{n}dxdv-2\int_{\R^{2d}}v\cdot z^{n-1}(t)f^{n}dxdv= I_{1}+I_{2}.
\end{equation*}

For $I_{1}$, we see that
\begin{equation*}
\begin{split}
I_{1} &= 2\int_{\R^{4d}}r(x-y)v\cdot(w-v)f^{n-1}(y,w,t)f^{n}(x,v,t)dydwdxdv\\
&\leq AM_{0}\int_{\R^{2d}}v^{2}f^{n-1}(x,v,t)dxdv-\psi(t)M_{0}\int_{\R^{2d}}v^{2}f^{n}dxdv.
\end{split}
\end{equation*}

For $I_{2}$, we have
\begin{equation*}
I_{2} = -2\int_{\R^{2d}}v\cdot z^{n-1}(t)f^{n}dxdv\leq\varepsilon\int_{\R^{2d}}v^{2}f^{n}dxdv+\frac{M_{0}M'^{2}}{\varepsilon}.
\end{equation*}

Therefore, we conclude that
\begin{equation*}
\small
\frac{d}{dt}\int_{\R^{2d}}v^{2}f^{n}dxdv\leq C\int_{\R^{2d}}v^{2}f^{n-1}(x,v,t)dxdv+(\varepsilon-\psi(t)M_{0})\int_{\R^{2d}}v^{2}f^{n}dxdv+\frac{C}{\varepsilon}.
\end{equation*}

We select $\varepsilon=\frac{\psi(t)M_{0}}{2}$, then we could get the uniform bound estimate on the kinetic energy of $f^{n}$.\\
$\bullet$($\vert\vert f^{n}\vert\vert_{L_{s}^{\infty}}$estimate):

To ensure that the equation still has a solution when the particle velocity is high, we introduce a weighted Sobolev space($s>0$):
\begin{equation*}
L_{s}^{\infty}:=\{f:\mathop{esssup}\limits_{(x,v)\in\R^{d}\times\R^{d}}(1+v^{2})^{s}\vert f(x,v)\vert<\infty\}.
\end{equation*}

Then we denote the $L_{s}^{\infty}$-norm as
\begin{equation}\label{4.10}
\vert\vert f\vert\vert_{L_{s}^{\infty}}:=\mathop{esssup}\limits_{(x,v)\in\R^{d}\times\R^{d}}(1+v^{2})^{s}\vert f(x,v)\vert.
\end{equation}

Then we have
\begin{equation*}
\begin{split}
\frac{d}{dt}(1+(v^{n})^{2})^{s}f^{n}(x,v) &=2sv^{n}(1+(v^{n})^{2})^{s-1}\frac{dv^{n}}{dt}f^{n}(x,v)+(1+(v^{n})^{2})^{s}\frac{df^{n}(x,v)}{dt}\\
&= J_{1}+J_{2}.
\end{split}
\end{equation*}

For $J_{1}$, we have
\begin{equation*}
\begin{split}
J_{1} =& 2sv^{n}(1+(v^{n})^{2})^{s-1}f^{n}(x,v)[F(f^{n-1})-z^{n-1}(t)]\\
\leq& 2sv^{n}(1+(v^{n})^{2})^{s-1}f^{n}(x,v)\int_{\R^{2d}}r(x-y)(w-v^{n})f^{n-1}(y,w,t)dydw\\
&-2sz^{n-1}(t)v^{n}(1+(v^{n})^{2})^{s-1}f^{n}(x,v)\\
\leq& C(1+(v^{n})^{2})^{s}f^{n}(x,v).
\end{split}
\end{equation*}

For $J_{2}$, we have
\begin{equation*}
\begin{split}
J_{2}&=-(1+(v^{n})^{2})^{s}f^{n}\nabla_{v}\cdot F(f^{n-1})= d(1+(v^{n})^{2})^{s}f^{n}\int_{\R{2d}}r(x-y)f^{n-1}dydw\\
&\leq AM_{0}d(1+(v^{n})^{2})^{s}f^{n}(x,v).
\end{split}
\end{equation*}

Therefore, we conclude that
\begin{equation*}
\frac{d}{dt}(1+(v^{n})^{2})^{s}f^{n}(x,v)\leq C(1+(v^{n})^{2})^{s}f^{n}(x,v).
\end{equation*}

By applying Gronwall's inequality, we have
\begin{equation}\label{4.11}
\vert\vert f^{n}\vert\vert_{L_{s}^{\infty}}\leq C\vert\vert f_{0}\vert\vert_{L_{s}^{\infty}}.
\end{equation}

Next, using the prior estimation above, we will show that $(1+(v^{n})^{2})^{s}f^{n}(x,v,t)$ is Cauchy in $L^{\infty}(0,T;L^{1}(\R^{d}\times\R^{d}))$. Multiplying $(1+v^{2})^{s}$ and integrating on both sides of the equation \eqref{4.1}, we have
\begin{equation*}
\begin{split}
\frac{d}{dt}&\int_{\R^{2d}}(1+v^{2})^{s}(f^{n+1}-f^{n})dxdv\\ =&\int_{\R^{2d}}(1+v^{2})^{s}\nabla_{v}\cdot[F(f^{n-1})f^{n}-F(f^{n})f^{n+1}]dxdv\\
&+\int_{\R^{2d}}(1+v^{2})^{s}(z^{n}(t)\cdot\nabla_{v}f^{n+1}-z^{n-1}(t)\cdot\nabla_{v}f^{n})dxdv\\
=& K_{1}+K_{2}.
\end{split}
\end{equation*}
where we suppose $(f^{n+1}-f^{n})>0$. If $(f^{n+1}-f^{n})<0$, add a negative sign on both sides to expand and contract in the same way.

For the estimate of $K_{1}$, we divide it into two terms
\begin{equation*}
\begin{split}
K_{1}=&2s\int_{\R^{2d}}v(1+v^{2})^{s-1}[F(f^{n})f^{n+1}-F(f^{n-1})f^{n}]dxdv\\
=& 2s\int_{\R^{2d}}v(1+v^{2})^{s-1}f^{n+1}(F(f^{n})-F(f^{n-1}))dxdv\\
&+2s\int_{\R^{2d}}v(1+v^{2})^{s-1}F(f^{n-1})(f^{n+1}-f^{n})dxdv\\
=& K_{11}+K_{12}.
\end{split}
\end{equation*}
where
\begin{equation*}
\begin{split}
K_{11} =&2s\int_{\R^{2d}}v(1+v^{2})^{s-1}f^{n+1}\int_{\R^{2d}}r(x-y)(w-v)(f^{n}-f^{n-1})dydwdxdv\\
\leq& s\int_{\R^{2d}}(1+v^{2})^{s-1}f^{n+1}\int_{\R^{2d}}r(x-y)(w)^{2}\vert f^{n}-f^{n-1}\vert dydwdxdv\\
&+s\int_{\R^{2d}}v^{2}(1+v^{2})^{s-1}f^{n+1}\int_{\R^{2d}}r(x-y)\vert f^{n}-f^{n-1}\vert dydwdxdv\\
\leq& C\int_{\R^{2d}}(1+v^{2})^{s}\vert f^{n}-f^{n-1}\vert dxdv.
\end{split}
\end{equation*}
\begin{equation*}
\begin{split}
K_{12} =&2s\int_{\R^{2d}}v(1+v^{2})^{s-1}(f^{n+1}-f^{n})\int_{\R^{2d}}r(x-y)(w-v)f^{n-1}(y,w,t)dydwdxdv\\
\leq& s\int_{\R^{2d}}(1+v^{2})^{s-1}(f^{n+1}-f^{n})\int_{\R^{2d}}r(x-y)(w)^{2}f^{n-1}(y,w,t)dydwdxdv\\
&-s\int_{\R^{2d}}v^{2}(1+v^{2})^{s-1}(f^{n+1}-f^{n})\int_{\R^{2d}}r(x-y)f^{n-1}(y,w,t)dydwdxdv\\
\leq& sAM\int_{\R^{2d}}(1+v^{2})^{s-1}(f^{n+1}-f^{n})dxdv,
\end{split}
\end{equation*}
where $M=\int_{\R^{2d}}v^{2}f^{n-1}dxdv$.

For the estimate of $K_{2}$, we use the upper bound of the memory term \eqref{4.7}
\begin{equation*}
\begin{split}
K_{2} =&2s\int_{\R^{2d}}v(1+v^{2})^{s-1}\cdot(z^{n-1}(t)f^{n}-z^{n}(t)f^{n+1})dxdv\\
=&-2sz^{n-1}(t)\int_{\R^{2d}}v(1+v^{2})^{s-1}(f^{n+1}-f^{n})dxdv\\
&+2s(z^{n-1}(t)-z^{n}(t))\int_{\R^{2d}}v(1+v^{2})^{s-1}f^{n}dxdv\\
\leq&2sM'\int_{\R^{2d}}(1+v^{2})^{s}(f^{n+1}-f^{n})dxdv+C\int_{\R^{2d}}(1+v^{2})^{s}\vert f^{n}-f^{n-1}\vert dxdv.
\end{split}
\end{equation*}

Hence, we can get
\begin{equation}\label{4.12}
\begin{split}
\frac{d}{dt}\int_{\R^{2d}}(1+v^{2})^{s}(f^{n+1}-f^{n})dxdv \leq& C\int_{\R^{2d}}(1+v^{2})^{s}(f^{n+1}-f^{n})dxdv\\
&+C\int_{\R^{2d}}(1+v^{2})^{s}\vert f^{n}-f^{n-1}\vert dxdv,
\end{split}
\end{equation}
where $C$ is a constant consists with $A$, $s$, $M$, $M'$ and $T$.

By Gronwall's inequality, we select a sufficiently small $T$ (satisfies $CTe^{CT}<1$), then $(1+v^{2})^{s}f^{n}$ is a Cauchy sequence in $L^{\infty}(0,T;L^{1}(\R^{d}\times\R^{d}))$. Thus, there exist a limiting function $(1+v^{2})^{s}f$.

Therefore, we only need to prove that the limiting function $f$ is the weak solution to the equation \eqref{2.12} in the sense of definition. For any $\eta\in \mathit{C}_{c}^{1}(\R^{d}\times\R^{d}\times[0,T])$ with $\eta(x,v,T)=0$,
\begin{equation*}
\small\int_{\R^{2d}}f_{0}\eta(x,v,0)dxdv+\int_{0}^{T}\int_{\R^{2d}}f^{n+1}(\partial_{t}\eta+v\cdot\nabla_{x}\eta+F(f^{n})\cdot\nabla_{v}\eta-z^{n}(t)\cdot\nabla_{v}\eta)dxdv=0
\end{equation*}

Since the remaining terms are linear, we only need to deal with the velocity alignment term and the memory term.

For velocity alignment term, there exists
\begin{equation*}
\begin{split}
&\int_{\R^{2d}}(f^{n+1}F(f^{n})-fF(f))\cdot\nabla_{v}\eta dxdv\\
&=\int_{\R^{2d}}(f^{n+1}F(f^{n})-f^{n+1}F(f))\cdot\nabla_{v}\eta dxdv+\int_{\R^{2d}}(f^{n+1}F(f)-fF(f))\cdot\nabla_{v}\eta dxdv\\
&=L_{1}+L_{2}.
\end{split}
\end{equation*}

For $L_{1}$, we have
\begin{equation*}
\begin{split}
L_{1} =&\int_{\R^{2d}}\int_{\R^{2d}}r(x-y)(w-v)(f^{n}-f)dydwf^{n+1}\cdot\nabla_{v}\eta dxdv\\
=&\int_{\R^{2d}}\int_{\R^{2d}}r(x-y)w(f^{n}-f)dydv'f^{n+1}\cdot\nabla_{v}\eta dxdv\\
&+\int_{\R^{2d}}\int_{\R^{2d}}r(x-y)(f^{n}-f)dydv'f^{n+1}v\cdot\nabla_{v}\eta dxdv\\
\to&0\;\;\;\;as\;\;\;\;n\to\infty.
\end{split}
\end{equation*}

For $L_{2}$, by using a similar argument, we have
\begin{equation*}
L_{2}\to0\;\;\;\;as\;\;\;\;n\to\infty.
\end{equation*}

For memory term, there exists
\begin{equation*}
\begin{split}
&\int_{0}^{T}\int_{\R^{2d}}(fz(t)\cdot\nabla_{v}\eta-f^{n+1}z^{n}(t)\cdot\nabla_{v}\eta)dxdv\\
&=\int_{0}^{T}\int_{\R^{2d}}(f-f^{n+1})z(t)\cdot\nabla_{v}\eta dxdv+\int_{0}^{T}\int_{\R^{2d}}f^{n+1}(z^{n}(t)-z(t))\cdot\nabla_{v}\eta dxdv\\
&\to0\;\;\;\;as\;\;\;\;n\to\infty.
\end{split}
\end{equation*}

Thus, we know that $(f,z)$ is the weak solution to the equation \eqref{2.12} in the sense of definition.

\hfill$\square$

\begin{remark}\label{remark 4.1}The above theorem states that in a sufficiently small amount of time $[0, T]$, the model has weak solution. By standard continuation argument, the solution can be extended step by step to any finite time interval.
\end{remark}

\section{The asymptotic behavior of mesoscopic states}
In this subsection, we will mainly discuss the asymptotic behavior of kinetic models \eqref{2.12}. Similar to our discussion in Section 3, we first provide estimates of the total momentum and kinetic energy of the system, and then estimate the velocity flocking rate and position flocking rate.

Let $M_{j}(t), j=1,2$ represent the total momentum and energy of the system \eqref{2.12},

\begin{equation*}
M_{1}(t) := \displaystyle\int_{\R^{2d}}vf(x,v,t)dxdv,\;\;M_{2}(t) := \displaystyle\int_{\R^{2}}v^{2}f(x,v,t)dxdv.
\end{equation*}

\begin{lemma}\label{Lemma 4.1}Let $(f,z)$ be the solution in Theorem \ref{theorem 2.6}. Then we have the following properties
\begin{equation}\label{5.1}
\frac{d}{dt}M_{1}(t)=-M_{0}z(t),
\end{equation}
\begin{equation}\label{5.2}
\frac{d}{dt}M_{2}(t)\leq2z(t)\cdot M_{1}(t).
\end{equation}
\begin{equation}\label{5.3}
\frac{d}{dt}M_{2}(t)\geq-2AM_{0}M_{2}(t)-2z(t)\cdot M_{1}(t)+2AM_{1}^{2}(t).
\end{equation}
\end{lemma}

\begin{proof}

For $M_{1}(t)$, we have
\begin{equation*}
\small
\begin{split}
\frac{d}{dt}M_{1}(t) = \int_{\R^{4d}}r(x-y)(w-v)f(y,w)f(x,v)dydwdxdv-z(t)\int_{\R^{2d}}fdxdv=-M_{0}z(t).
\end{split}
\end{equation*}

Next, for $M_{2}(t)$, we have
\begin{equation*}
\begin{split}
\frac{d}{dt}M_{2}(t) &= -2z(t)\int_{\R^{2d}}vfdxdv-\int_{\R^{4d}}r(x-y)\vert w-v\vert^{2}f(y,w)f(x,v)dydwdxdv\\
&\leq -2z(t)\cdot M_{1}(t).
\end{split}
\end{equation*}

Moreover, we have
\begin{equation*}
\begin{split}
\frac{d}{dt}M_{2}(t) &= -\int_{\R^{4d}}r(x-y)\vert w-v\vert^{2}f(y,w)f(x,v)dydwdxdv-2z(t)\int_{\R^{2d}}vfdxdv\\
&\geq-A\int_{\R^{4d}}\vert w-v\vert^{2}f(y,w)f(x,v)dydwdxdv-2z(t)\int_{\R^{2d}}vfdxdv\\
&\geq-2AM_{0}M_{2}(t)-2z(t)M_{1}(t)+2AM_{1}^{2}(t).
\end{split}
\end{equation*}
\end{proof}

\begin{lemma}\label{lemma 5.2}Let $(f,z)$ be the solution in Theorem \ref{theorem 2.6}. When $\lambda>1$, the memory term $z(t)$ has the following form
\begin{equation}\label{5.4}
z(t)=\overline{C}_{1}e^{\overline{\mu}_{1}t}+e^{\overline{\alpha}t}(\overline{C}_{2}\cos(\beta t)+\overline{C}_{3}\sin(\beta t)),
\end{equation}
where $\overline{\mu}_{1}<0$, $\overline{\alpha}<0$, $\beta>0$, $\overline{C}_{1}$, $\overline{C}_{2}$, $\overline{C}_{3}$ are constants.
\end{lemma}

\begin{proof}
Combining the previously obtained equation \eqref{2.12} and \eqref{5.1}, we have
\begin{equation}\label{5.5}
\begin{cases}
\frac{d\overline{x}}{dt}=M_{1}(t),\\
\frac{dM_{1}}{dt}=-M_{0}z(t),\\
\frac{dz}{dt}=\overline{x}(t)+\overline{v}(t)-\lambda z(t).
\end{cases}
\end{equation}

Therefore, we can obtain the characteristic equation about $\bar{x}$
\begin{equation}\label{5.6}
\mu^{3}+\lambda\mu^{2}+M_{0}\mu+M_{0}=0.
\end{equation}

By using Routh-Hurwitz criterion, when $\lambda>1$ the solution of the characteristic equation is
\begin{equation*}
\overline{\mu}_{1}<0,\;\;\overline{\mu}_{2}=\overline{\alpha}+i\overline{\beta},\;\;\overline{\mu_{3}}=\overline{\alpha}-i\overline{\beta},
\end{equation*}
where $\overline{\alpha}<0$ and $\overline{\beta}>0$.

Then we can obtain the solution of $\overline{x}(t)$
\begin{equation}\label{5.7}
\overline{x}(t)=C_{4}e^{\overline{\mu}_{1}t}+e^{\overline{\alpha} t}(C_{5}\cos(\overline{\beta} t)+C_{6}\sin(\overline{\beta} t)),
\end{equation}
where $\overline{\mu}_{1}<0$, $\overline{\alpha}<0$, $\overline{\beta}>0$, $C_{4}$, $C_{5}$, $C_{6}$ are constants.

Combined with the equation \eqref{5.5}, we have
\begin{equation}\label{5.8}
z(t)=\overline{C}_{4}e^{\overline{\mu}_{1}t}+e^{\overline{\alpha} t}(\overline{C}_{5}\cos(\beta t)+\overline{C}_{6}\sin(\beta t)),
\end{equation}
where $\mu_{1}<0$, $\overline{\alpha}<0$, $\beta>0$, $\overline{C}_{4}$, $\overline{C}_{5}$, $\overline{C}_{6}$ are constants.
\end{proof}

\begin{lemma}\label{lemma 5.4}Let $(f,z)$ be the global solution in Theorem \ref{theorem 2.6}. Then the $v(t)$ satisfies the following estimates
\begin{align}
&v(t)\geq v(0)e^{-AM_{0}t}+\int_{0}^{t}e^{-AM_{0}(t-s)}(z(s)-A\sqrt{M_{0}M_{2}(s)})ds,\\
&v(t)\leq v(0)e^{-M_{0}\int_{0}^{t}\psi(s)ds}+\int_{0}^{t}e^{-M_{0}\int_{s}^{t}\psi(\tau)d\tau}(A\sqrt{M_{0}M_{2}(s)}+z(s))ds,
\end{align}
where $\psi(t)=\displaystyle\min_{(x(t),y(t))} r(x-y)$.
\end{lemma}

\begin{proof}
Let $(x(t),v(t))$ be the particle trajectory issued from $(x,v)\in \displaystyle{supp}_{(x,v)}f_{0}$ at time 0.

We set $x(t)=x(t;0,x,v),v(t)=v(t;0,x,v)$, then we have
\begin{equation*}
x'(t)=v(t),\;\;v'(t)=F(f)(x(t),v(t),t))-z(t).
\end{equation*}

For $v(t)$, there exists
\begin{equation*}
\begin{split}
\frac{d}{dt}v(t) &= F(f)(x(t),v(t),t))-z(t)=\int_{\R^{2d}}r(x-y)(w-v)f(y,w,t)dydw-z(t)\\
&=\int_{\R^{2d}}r(x-y)wf(y,w,t)dydw-v(t)\int_{\R^{2d}}r(x-y)f(y,w,t)dydw-z(t).
\end{split}
\end{equation*}

Then we have
\begin{equation*}
-A\sqrt{M_{0}M_{2}(t)}-AM_{0}v(t)-z(t)\leq\frac{d}{dt}v(t)\leq-M_{0}\psi(t)v(t)+A\sqrt{M_{0}M_{2}(t)}-z(t).
\end{equation*}

By using Gronwall's inequality, we have
\begin{align*}
&v(t)\geq v(0)e^{-AM_{0}t}-\int_{0}^{t}e^{-AM_{0}(t-s)}(z(s)+A\sqrt{M_{0}M_{2}(s)})ds,\\
&v(t)\leq v(0)e^{-M_{0}\int_{0}^{t}\psi(s)ds}+\int_{0}^{t}e^{-M_{0}\int_{s}^{t}\psi(\tau)d\tau}(A\sqrt{M_{0}M_{2}(s)}-z(s))ds.
\end{align*}
\end{proof}

\noindent\textbf{Proof of Proposition \ref{proposition 2.7}.}
The expression for the average memory is
\begin{equation*}
z(t)=\overline{C}_{4}e^{\overline{\mu}_{1}t}+e^{\overline{\alpha} t}(\overline{C}_{5}\cos(\beta t)+\overline{C}_{6}\sin(\beta t)).
\end{equation*}

By using \eqref{5.1}, \eqref{5.2}, \eqref{5.8} and the expression of velocity $v(t)$, we can obtain
\begin{equation*}
\begin{split}
\vert v\vert \leq\vert v(0)\vert+A\sqrt{M_{0}}\int_{0}^{t}\sqrt{M_{2}(s)}ds+\int_{0}^{t}z(s)ds\leq\vert v(0)\vert+\zeta t^{2},
\end{split}
\end{equation*}
where $\zeta$ is a constant depend on $A$, $M_{0}$, $M_{1}(0)$ and $M_{2}(0)$.

And for $x(t)$, we have
\begin{equation*}
\vert x(t)\vert\leq\vert x(0)\vert+\int_{0}^{t}\vert v(s)\vert ds\leq\vert x(0)\vert+\vert v(0)\vert t+\zeta t^{3}.
\end{equation*}

For $\psi(t)$, there are
\begin{equation}\label{5.11}
\psi(t)\geq\frac{A}{(1+\vert x-y\vert^{2})^{\beta}}\geq\frac{A}{(1+4x^{2})^{\beta}}\geq \frac{A}{\overline k_{0}}(1+t^{2}+t^{6})^{-\beta},
\end{equation}
where $\overline k_{0}$ is a constant depend on $\zeta$, $x(0)$ and $v(0)$.
\hfill$\square$

\begin{lemma}\label{lemma 5.5}Let $(f,z)$ be the global solution in Theorem \ref{theorem 2.6}. Then we have
\begin{equation}\label{5.12}
\frac{d}{dt}\Lambda[f(t)]=-\int_{\R^{4d}}r(x-y)\vert w-v\vert^{2}f(y,w)f(x,v)dydwdxdv.
\end{equation}
\end{lemma}

\begin{proof}
By calculation, it can be obtained
\begin{equation*}
\begin{split}
\frac{d}{dt}\Lambda[f(t)] &= \int_{\R^{2d}}\vert v-u_{c}\vert^{2}\partial_{t}f(x,v,t)dxdv\\
&=\int_{\R^{2d}}\vert v-u_{c}\vert^{2}\nabla_{v}\cdot[(z(t)-F(f))f]dxdv-\int_{\R^{2d}}\vert v-u_{c}\vert^{2}v\cdot\nabla_{x}fdxdv\\
&=2\int_{\R^{2d}}(v-u_{c})F(f)fdxdv-2z(t)\cdot\int_{\R^{2d}}(v-u_{c})fdxdv\\
&=I_{1}+I_{2}.
\end{split}
\end{equation*}

The first term is simplified as follows.
\begin{equation*}
I_{1}=-\int_{\R^{4d}}r(x-y)\vert w-v\vert^{2}f(y,w)f(x,v)dydwdxdv.
\end{equation*}

For the second term, we have
\begin{equation*}
I_{2}=-2z(t)\cdot\int_{\R^{2d}}(v-u_{c})fdxdv=0.
\end{equation*}

Then there exists
\begin{equation*}
\frac{d}{dt}\Lambda[f(t)]=-\int_{\R^{4d}}r(x-y)\vert w-v\vert^{2}f(y,w)f(x,v)dydwdxdv.
\end{equation*}
\end{proof}

Based on the Lemma \ref{lemma 5.5} obtained above, we can derive the following theorem.

\noindent\textbf{Proof of Theorem \ref{theorem 2.9}.}
From the Lemma \ref{lemma 5.5}, it can be concluded that
\begin{equation*}
\begin{split}
\frac{d}{dt}\Lambda[f(t)] &=-\int_{\R^{4d}}r(x-y)\vert w-v\vert^{2}f(y,w)f(x,v)dydwdxdv\\
&\leq-\psi(t)\int_{\R^{4d}}\vert w-v\vert^{2}f(y,w)f(x,v)dydwdxdv=-\psi(t)M_{0}\Lambda[f(t)].
\end{split}
\end{equation*}

By Gronwall's inequality, we have
\begin{equation*}
\Lambda[f(t)]\leq\Lambda[f_{0}]e^{-\int_{0}^{t}\psi(s)ds}.
\end{equation*}

By using Proposition \ref{proposition 2.7} to discuss the lower bound of $\psi(t)$, we can obtain
\begin{equation*}
\psi(t)\geq\frac{A}{\overline\nu_{0}}t^{-6\beta},\;\;t>1,
\end{equation*}
where $\overline\nu_{0}=3^{\beta}\overline k_{0}$.

Then we have
\begin{equation}\label{5.13}
\Lambda[f(t)]\leq
\begin{cases}
\Lambda[f_{0}]e^{-\gamma_{1}t^{1-6\beta}},\;\;\;\;\beta\in[0,\frac{1}{6}),\\
\Lambda[f_{0}](1+t)^{-\gamma_{2}},\;\;\;\;\beta=\frac{1}{6},
\end{cases}
\end{equation}
where $\gamma_{1}=\frac{A}{(1-6\beta)\overline\nu_{0}}$, $\gamma_{2}=\frac{A}{\overline\nu_{0}}$.
\hfill$\square$

\section{Conclusions}

This paper provides C-S model with memory term for understanding the collective behavior of biological populations with memory abilities, such as fish schools and migratory animals. The results show that memory is the key factor affecting flocking behavior. Memory makes flocking harder. For a population system with convolutional memory, we find the memory can destroy the flocking phenomenon both in the microscopic level and mesoscopic level. Our work lays a foundation for the further research on the relationship between more complex biological cognitive behaviors (such as learning and decision-making) and collective movements.

We first innovate the model by introducing a continuous memory term described by an internal state variable $h(t)$ in the Cucker-Smale framework, whose dynamic equation is $\frac{dh}{dt}=x(t)+v(t)-\lambda h(t)$. The specific form of $h(t)$ is non-local convolution form in \cite{Xie2020Convolution}. This is different from previous models that use fixed time delays and can more accurately reflect the continuous dependence of biology on past paths, past velocity and the exponential decay characteristics of memory. At the microscopic level, we construct a coupled ODE system of position $x(t)$, velocity $v(t)$ and memory $h(t)$. Memory terms act as the internal driving force or biologies and affect their movement. At the mesoscopic level, we derive a Vlasov-type kinetic equation with memory term by using mean-field limit. Describing the distribution pattern of population density at the mesoscopic level through a system of partial differential equations with memory.

Then at microscopic level, we find that the memory term has a destroy effect on flocking phenomenon. Even if the center position $x^{c}(t)$, center velocity $v^{c}(t)$ and center memory $h^{c}(t)$ reach stability. It requires the communication kernel has a polynomial lower bound form. And the memory deviation needs to satisfy the integrability condition as \cite{Brown2003Social}.

Finally, we discuss the flocking behavior at the mesoscopic level. We find that compared with classical C-S model, the range of communication strength $\beta$ become smaller. With the addition of memory term, the occurrence of flocking phenomenon requires memory decay strong enough without the extra conditions of $\psi(t)$ and $H(t)$. Memory makes flocking become difficult. This consistency is coincide with biological observations: the flocking behavior of groups(such as bird flocks \cite{Odling2003influence} and fish schools \cite{Brown2003Social,Biro2016Bringing}) is controlled by individual memory and overall population constraints. We have only studied the impact of memory on biological flocking phenomenon at the microscopic and mesoscopic levels, and it is also necessary to study flocking phenomenon at the macroscopic level.

\section*{Acknowledgments}
The authors are grateful to the referees for their careful reading and valuable comments.


\begin{thebibliography}{00}
\bibitem{Brown2003Social}
{\sc C. Brown and K. N. Laland}, {\em Social learning in fishes: a review},
  Fish and Fisheries, 4 (2003), pp.~280--288,
  \url{https://doi.org/10.1046/j.1467-2979.2003.00122.x}.

\bibitem{Brown2008}
{\sc H. R. Brown and W. Myrvold}, {\em Boltzmann's h-theorem, its limitations,
  and the birth of (fully) statistical mechanics}, 2008,
  \url{https://doi.org/10.48550/arXiv.0809.1304}.

\bibitem{Caplan2019}
{\sc J. B. Caplan and K. Kato}, {\em The brain's representations
  may be compatible with convolution-based memory models}, Memory, 27 (2017),
  pp.~299-312, \url{https://doi.org/10.1037/cep0000115}.

\bibitem{Choi2024From}
{\sc Y. P. Choi and B. H. Hwang}, {\em From {BGK}-alignment model to the
  pressured {E}uler-alignment system with singular communication weights}, J.
  Differential Equations, 379 (2024), pp.~363--412,
  \url{https://doi.org/10.1016/j.jde.2023.10.010}.

\bibitem{CS2007}
{\sc F. Cucker and S. Smale}, {\em Emergent behavior in flocks}, IEEE
  Transactions on Automatic Control, 52 (2007), pp.~852--862,
  \url{https://doi.org/10.1109/TAC.2007.895842}.

\bibitem{Cucker2007On}
{\sc F. Cucker and S. Smale}, {\em On the mathematics of emergence}, Japanese
  Journal of Mathematics, 2 (2007), pp.~197--227,
  \url{https://doi.org/10.1007/s11537-007-0647-x}.

\bibitem{Benedetto1999A}
{\sc E. Caglioti, D.~Benedetto and M.~Pulvirenti}, {\em A kinetic equation for
  granular media}, RAIRO - Mathematical Modelling and Numerical Analysis, 31
  (1997), pp.~615--641, \url{https://www.numdam.org/item/M2AN_1997__31_5_615_0/}.

\bibitem{Biro2016Bringing}
{\sc T. Sasaki, D.~Biro and S. J. Portugal}, {\em Bringing a time-depth perspective
  to collective animal behaviour}, Trends in Ecology \& Evolution, 31 (2016),
  pp.~550--562,
  \url{https://doi.org/10.1016/j.tree.2016.03.018}.

\bibitem{Degond2008Continuum}
{\sc P.~Degond and S.~Motsch}, {\em Continuum limit of self-driven particles
  with orientation interaction}, Mathematical Models and Methods in Applied
  Sciences, 18 (2008), pp.~1193--1215,
  \url{https://hal.science/hal-00175719}.

\bibitem{Degond2008Large}
{\sc P.~Degond and S.~Motsch}, {\em Large scale dynamics of the persistent
  turning walker model of fish behavior}, Journal of Statistical Physics, 131
  (2008), pp.~989--1021, \url{https://doi.org/10.1007/s10955-008-9529-8}.

\bibitem{Golse2016On}
{\sc F.~Golse}, {\em On the dynamics of large particle systems in the mean
  field limit}, Macroscopic and Large Scale Phenomena: Coarse Graining, Mean
  Field Limits and Ergodicity, 3 (2016), pp.~1--144,
  \url{https://doi.org/10.1007/978-3-319-26883-5\_1}.

\bibitem{Carrillo2010Particle}
{\sc G. Toscani, J. A. Carrillo, M. Fornasier and F. Vecil}, {\em Particle, kinetic,
  and hydrodynamic models of swarming}, Mathematical Modeling of Collective
  Behavior in Socio-Economic and Life Sciences,  (2010), pp.~297--336,
  \url{https://doi.org/10.1007/978-0-8176-4946-3\_12}.

\bibitem{Carrillo2010Asymptotic}
{\sc J. Rosado, J. A. Carrillo, M. Fornasier and G. Toscani}, {\em Asymptotic
  flocking dynamics for the kinetic cucker-smale model}, SIAM Journal on
  Mathematical Analysis, 42 (2010), pp.~218--236,
  \url{https://doi.org/10.1137/090757290}.

\bibitem{Carrillo2017A}
{\sc Y. P. Choi, J. A. Carrillo and S. P. P\'erez}, {\em A review on
  attractive-repulsive hydrodynamics for consensus in collective behavior},
  Active Particles, Volume 1: Advances in Theory, Models, and Applications,
  (2017), \url{https://doi.org/10.1007/978-3-319-49996-3_7}.

\bibitem{Parrish2002Self}
{\sc S. V. Viscido J. K. Parrish and D. Gr\"unbaum}, {\em Self-organized fish
  schools: An examination of emergent properties}, Biological Bulletin, 202
  (2002), pp.~296--305, \url{https://doi.org/10.2307/1543482}.

\bibitem{Kennard1938}
{\sc E.~H. Kennard}, {\em Kinetic theory of gases}, McGraw-Hill Book Company,
  New York and London,  (2025),
  \url{https://www.britannica.com/science/Brownian-motion}.

\bibitem{He2017Feedback}
{\sc X. Liang, J. Zhang L. He, P. Bai and W. Wang}, {\em Feedback formation control of
  uav swarm with multiple implicit leaders}, Aerospace Science and Technology,
  72 (2018), pp.~327--334,
  \url{https://doi.org/10.1016/j.ast.2017.11.020}.

\bibitem{Motsch2011A}
{\sc S. Motsch and E. Tadmor}, {\em A new model for self-organized dynamics and
  its flocking behavior}, Journal of Statistical Physics, 144 (2011),
  pp.~923--947,
  \url{https://doi.org/10.1007/s10955-011-0285-9}.

\bibitem{Odling2003influence}
{\sc L. Odling-Smee and V. A. Braithwaite}, {\em The influence of habitat
  stability on landmark use during spatial learning in the three-spined
  stickleback}, Animal Behaviour, 65 (2003), pp.~701--707,
  \url{https://doi.org/10.1006/anbe.2003.2082}.

\bibitem{Pignotti2017Flocking}
{\sc C.~Pignotti and I. R. Vallejo}, {\em Flocking estimates for the
  cucker-smale model with time lag and hierarchical leadership}, Journal of
  Mathematical Analysis and Applications, 464 (2018), pp.~1313--1332,
  \url{https://doi.org/10.1016/j.jmaa.2018.04.070}.

\bibitem{Poyato2016Euler}
{\sc D. Poyato and J. Soler}, {\em Euler-type equations and commutators in
  singular and hyperbolic limits of kinetic cucker-smale models}, Mathematical
  Models and Methods in Applied Sciences, 27 (2017), pp.~1089--1152,
  \url{https://doi.org/10.1142/S0218202517400103}.

\bibitem{Russo1996Kinetic}
{\sc G. Russo and P. Smereka}, {\em Kinetic theory for bubbly flow ii: Fluid
  dynamic limit}, SIAM Journal on Applied Mathematics, 56 (1996), pp.~358--371,
  \url{https://api.semanticscholar.org/CorpusID:43490226}.

\bibitem{Vicsek2006Novel}
{\sc E. Ben-Jacob, I. Cohen, T. Vicsek, A. Czirok and O. Sochet}, {\em Novel type of
  phase transition in a system of self-driven particles}, Physical Review
  Letters, 75 (2006), pp.~1226--1229,
  \url{https://doi.org/10.1103/PhysRevLett.75.1226}.

\bibitem{Tadmor2017From}
{\sc E. Tadmor and S. Y. Ha}, {\em From particle to kinetic and hydrodynamic
  descriptions of flocking}, Kinetic and Related Models, 1 (2017),
  pp.~415--435, \url{https://doi.org/10.3934/krm.2008.1.415}.

\bibitem{Toner1998Flocks}
{\sc J. Toner and Y. Tu}, {\em Flocks, herds, and schools: A quantitative
  theory of flocking}, Physical Review E, 58 (1998), pp.~4828--4858,
  \url{https://doi.org/10.1103/PhysRevE.58.4828}.

\bibitem{Topaz2004Swarming}
{\sc C. M. Topaz and A. L. Bertozzi}, {\em Swarming patterns in a
  two-dimensional kinematic model for biological groups}, SIAM Journal on
  Applied Mathematics, 65 (2004), pp.~152--174,
  \url{https://doi.org/10.1137/S0036139903437424}.

\bibitem{Xie2020Convolution}
{\sc Y. Chen, Y. Xie and M. Li}, {\em Convolution forgetting curve model for
  repeated learning}, 2020 International Conference on Artificial Intelligence
  and Education (ICAIE),  (2020), pp.~447--454,
  \url{https://doi.org/10.1109/ICAIE50891.2020.00109}.

\bibitem{Choi2017Emergent}
{\sc Y. P. Choi, S. Y. Ha and Z. Li}, {\em Emergent dynamics of the cucker-smale
  flocking model and its variants}, Active Particles, Volume 1: Advances in
  Theory, Models, and Applications,  (2017), pp.~299--331,
  \url{https://doi.org/10.1007/978-3-319-49996-3_8}.













\end{thebibliography}
\end{document}